\documentclass[12pt]{article}
\usepackage{amsthm,amsfonts,amsmath}
\usepackage{fullpage}

\usepackage{hyperref}

\newtheorem{thm}{Theorem}

\newtheorem{lem}{Lemma}

\newtheorem{dfn}{Definition}

\newcommand{\C}{\mathbb{C}}
\newcommand{\<}{\langle}
\renewcommand{\>}{\rangle}

\newcommand{\cU}{\mathcal{U}}
\newcommand{\cV}{\mathcal{V}}

\begin{document}
\title{An inertial  upper bound for the quantum independence number of a graph}

\author{Pawel Wocjan\thanks{\texttt{wocjan@cs.ucf.edu}, Department of Computer Science, University of Central Florida, USA}  \quad Clive Elphick\thanks{\texttt{clive.elphick@gmail.com}, School of Mathematics, University of Birmingham, Birmingham, UK}}

\maketitle

\begin{abstract}
A  well known upper bound for the independence number $\alpha(G)$ of  a graph $G$, is that 
\[
\alpha(G) \le n^0 + \min\{n^+ , n^-\},
\]
where $(n^+, n^0, n^-)$ is the inertia of $G$. We prove that this bound is also an upper bound for the quantum independence number $\alpha_q$(G), where $\alpha_q(G) \ge \alpha(G)$. We identify numerous graphs for which $\alpha(G) = \alpha_q(G)$ and demonstrate that there are graphs for which the above bound is not exact with any Hermitian weight matrix, for $\alpha(G)$ and $\alpha_q(G)$.  This result complements results by the authors that many spectral lower bounds for the chromatic number are also lower bounds for the quantum chromatic number.

\end{abstract}

\section{Introduction}

Wocjan and Elphick \cite{wocjan18} proved that many spectral lower bounds for the chromatic number, $\chi(G)$, are also lower bounds for the quantum chromatic number, $\chi_q(G)$. This was achieved using pinching and twirling and a combinatorial definition of $\chi_q(G)$ due to Mancinska and Roberson \cite{mancinska16}.

In a different paper Mancinska and Roberson \cite{mancinska162} defined a quantum independence number $\alpha_q(G)$, using quantum homomorphisms. It is known (see for example Section 6.18 of \cite{roberson13}) that:
\[
\alpha(G) \le \alpha_q(G) \le \lfloor \theta'(G) \rfloor \le \theta'(G) \le \theta(G) \le \theta^+(G) \le \lceil \theta^+(G) \rceil \le \chi_q(\overline{G}) \le \chi(\overline{G})\,,
\]
where $\theta', \theta, \theta^+$ are the Schrijver, Lov\'asz and Szegedy theta functions. 

Analogously to $\chi_q(G)$, $\alpha_q(G)$ is the maximum integer $t$ for which two players sharing an entangled state can convince a referee that the graph $G$ has an independent set of size $t$. There exist graphs $G$ for which there is an exponential separation between $\alpha(G)$ and $\alpha_q(G)$ \cite{mancinska162}. 

\section{Inertial  upper bound for the independence number}

The best known spectral upper bound for $\alpha(G)$ is as follows:

\begin{equation}\label{eq:alpha}
\alpha(G) \le n^0(W) + \min\{n^+(W) , n^-(W)\},
\end{equation}
where $W$ is a Hermitian weighted adjacency matrix of $G$ and $n^0, n^+, n^-$ are the numbers of zero, positive and negative eigenvalues of $W$ respectively. We also let $A$ denote the adjacency matrix of $G$, and $V(G)$ and $E(G)$ denote the vertex set and edge set of $G$ respectively, where $|V(G)| = n$ and $|E(G)| = m$. Obviously $n = n^0 + n^- + n^+$ and $rank(G) = n^+ + n^-$.

This bound was originally proved by Cvetkovic \cite{cvetkovic73}, using interlacing of the eigenvalues of the empty graph on $\alpha(G)$ vertices.  Godsil \cite{godsil06} presented an alternative proof that we generalize to prove the inertial upper bound on the quantum independence number.  His proof relies on the following elementary result:

\begin{lem}\label{lem:max_iso}
Let $M\in \C^{s\times s}$ be an arbitrary Hermitian matrix.  A subspace $\mathcal{U}$ of $\C^s$ is called totally isotropic with respect to the Hermitian form defined by $M$ if
\[
\< \Psi | M | \Psi\> = 0
\]
for all vectors $|\Psi\>\in\mathcal{U}$.  The dimension of all maximal totally isotropic subspaces is equal to 
\begin{equation}
\label{eq:dimension}
n^0(M) + \min\{ n^+(M), n^-(M) \}\,.
\end{equation}
\end{lem} 

\begin{proof}
Using Sylvester's law of inertia we may assume that $M$ is diagonal and has only eigenvalues $+1$, $0$, and $-1$. It is easy to show that there exists a totally isotropic subspace having dimension as in (\ref{eq:dimension}). Let $|\varphi^+_{a}\>$, $|\varphi^0_{b}\>$, $|\varphi^-_{c}\>$ denote the eigenvectors and the corresponding eigenvalues $+1$, $0$ and $-1$, respectively, where $a\in[n^+], b\in [n^0]$, and $c\in [n^-]$.  We may assume without loss of generality that $n^+\ge n^-$.
Then the vectors
\begin{align*}
|\varphi^0_b\>\,,                                 & \quad b\in [n^0]\,, \\
|\varphi^{+}_c\> + |\varphi^{-}_c\>\,, & \quad c\in [n^-]
\end{align*}
span a totally isotropic subspace of dimension $n^0+n^-$.  

We now show that there cannot exist a totally isotropic subspace whose dimension is larger than $n^0+n^-$. Let $\mathcal{U}$ be any totally isotropic subspace. Let $\mathcal{V}$ be the subspace spanned by $|\varphi^+_a\>$ for $a\in [n^+]$.  We have
\begin{eqnarray*}
n^+ + n^0 + n^- 
& = &
s \\
& \ge & 
\dim(\cU + \cV) \\
& =   &
\dim(\cU) + \dim(\cV) + \dim(\cU \cap \cV) \\
& = &
\dim(\cU) + n^+\,,
\end{eqnarray*}
which shows that $\dim(\cU)$ cannot be larger than $n^0 + n^-$.
\end{proof}

\begin{proof}
The inertial upper bound on $\alpha$ is established as follows.
Let $G$ be a graph with vertex set $V$ and Hermitian weighted adjacency matrix $W$. Let $e_u$ denote the standard basis vector corresponding to vertex $u$. If $S$ is an independent set in $V$ and $u, v \in S$, then

\[
\<e_u | W | e_v\> = 0.
\]

It follows that the subspace spanned by the orthogonal vectors $e_u$ for $u$ in $S$ is a totally isotropic subspace. The dimension of such a subspace is bounded by the inertia of $W$, as shown in Lemma~\ref{lem:max_iso}. 
\end{proof}

\section{Inertial upper bound for the quantum independence number}

\begin{thm}\label{thm:inertial_bound}
For any graph $G$ with quantum independence number $\alpha_q(G)$ and Hermitian weighted adjacency matrix $W$:

\begin{equation}\label{eq:alphaq}
\alpha_q(G) \le n^0(W) + \min\{n^+(W) , n^-(W)\}.
\end{equation}
\end{thm}

In order to prove Theorem 1 we need a combinatorial definition of $\alpha_q$. Laurent and Piovesan provide a combinatorial definition of $\alpha_q$ in Definition 2.5 of \cite{laurent15}, but we use the equivalent definition described in the paragraph after their Definition 2.8. 

For matrices $X,Y\in\C^{d\times d}$, their trace inner product (also called Hilbert-Schmidt inner product) is defined as
\[
\<X, Y\>_{\mathrm{tr}} = \mathrm{tr}(X^\dagger Y)\,.
\]

\begin{dfn}[Quantum independence number $\alpha_q$]\label{def:alpha_q}
For a graph $G$, $\alpha_q(G)$ is the maximum integer $t$ for which there exist orthogonal projectors $P^{(u,i)} \in\C^{d\times d}$ for $u \in V(G), i \in[t]$ satisfying the following conditions:
\begin{align}
\sum_{u\in V(G)} P^{(u,i)} = I_d    & \quad \mbox{for all } i \in [t]\quad & \label{eq:complete} \\
\<P^{(u,i)} , P^{(u,j)}\>_{\mathrm{tr}} = 0  & \quad \mbox{for all } i \not= j \in [t], \mbox{ for all } u \in V(G) & \label{eq:orthogonal u} \\
\<P^{(u,i)} , P^{(v,j)}\>_{\mathrm{tr}} = 0  & \quad \mbox{for all } i \not= j \in [t], \mbox{ for all } uv \in E(G).\quad & \label{eq:orthogonal uv}
\end{align}
We refer to condition (\ref{eq:complete}) as the \emph{completeness} condition and to conditions (\ref{eq:orthogonal u}) and (\ref{eq:orthogonal uv}) as the \emph{orthogonality} conditions.\footnote{Deviating slightly from \cite{laurent15}, we present the cases $u=v$ and $uv\in E$ as separate conditions to simplify the presentation below.}
\end{dfn}

Observe that the (classical) independence number $\alpha(G)$ is a special case of $\alpha_q(G)$ when the dimension $d$ is restricted to be $1$, that is, the only possible ``projectors'' are the scalars $1$ and $0$. More precisely, let $S=\{u_i \mid i \in [t]\}$ be any (classical) independent set. Then, for all $i\in [t]$, we set $P^{(u_i,i)}=1$ and $P^{(v,i)}=0$ for all $v\in V$ with $v\neq u_i$.

Mancinska \emph{et al} \cite{mancinska161} defined the projective packing number, $\alpha_p(G)$, as follows, and noted that for all graphs $\alpha_q(G) \le \alpha_p(G)$.
For the sake of completeness, we include the simple proof showing that the quantum independence number is bounded from above by the projective packing number. We show afterwards that the projective packing number is bounded from above by the inertia bound. It is more convenient to work with the projective packing number than with the quantum independence number.

\begin{dfn}
A $d$-dimensional projective packing of a graph $G=(V,E)$ is a collection of orthogonal projectors $P^{(u)}\in\C^{d\times d}$ such that
\begin{equation}\label{eq:projective_orthogonal}
\< P^{(u)}, P^{(v)} \>_{\mathrm{tr}} = 0
\end{equation}
for all $uv\in E$.
The value of a projective packing using projectors $P^{(u)}\in\C^{d\times d}$ is defined as
\begin{equation}
\frac{1}{d} \sum_{u\in V} r^{(u)}\,,
\end{equation}
where $r^{(u)}$ denote the ranks of the operators $r^{(u)}$.
The projective packing number $\alpha_p(G)$ of a graph $G=(V,E)$ is defined as the supremum of the values over all projective packings of the graph $G$.
\end{dfn}

\begin{lem}
For all graphs, we have $\alpha_q(G)\le \alpha_p(G)$.
\end{lem}

\begin{proof}
Let $P^{(u,i)}\in\C^{d\times d}$, $u\in V, i \in [t],$ be a collection of orthogonal projectors satisfying the conditions in Definition~\ref{def:alpha_q}. Define the operators
\begin{equation}
P^{(u)} = \sum_{i\in [t]} P^{(u,i)}\,.
\end{equation}
These operators are orthogonal projectors because of condition~(\ref{eq:orthogonal u}).  For $uv\in E$, we have
\begin{eqnarray}
\< P^{(u)}, P^{(v)}\>_{\mathrm{tr}} & = & \sum_{i,j\in[t]} \< P^{(u,i)}, P^{(v,j)}\>_{\mathrm{tr}} \\
& = &
\sum_{i\neq j\in[t]} \< P^{(u,i)}, P^{(v,j)}\>_{\mathrm{tr}} + \sum_{i\in[t]} \< P^{(u,i)}, P^{(v,i)}\>_{\mathrm{tr}}  \label{eq:middle} \\
& = &
0\,.
\end{eqnarray}
The sums in (\ref{eq:middle}) are equal to zero because of conditions~(\ref{eq:orthogonal uv}) and (\ref{eq:complete}), respectively.
We have
\[
\sum_{u\in V} r^{(u)} = \sum_{i\in [t]} \sum_{u\in V} \mathrm{rank}(P^{(u,i)}) = \sum_{i\in [t]} d = t \cdot d\,,
\]
because the projectors $P^{(u,i)}$ for each $i\in[t]$ add up to $I_d$ due to condition (\ref{eq:complete}). This concludes the proof $\alpha_q(G)\le \alpha_p(G)$.
\end{proof}

We will use the following result to reformulate the conditions on the orthogonal projectors of a projective packing as conditions on their eigenvectors.  We omit the proof of this basic result.
\begin{lem}\label{lem:elementary}
Let $P,Q\in\C^{d\times d}$ be two arbitrary orthogonal projectors of rank $r$ and $s$, respectively. Let
\[
P = \sum_{k\in[r]} |\psi_k\>\<\psi_k| \quad\mbox{and}\quad Q = \sum_{\ell\in[s]} |\phi_\ell\>\<\phi_\ell|
\]
denote their spectral resolutions, respectively.  Then, the following two conditions are equivalent:
\begin{align}
\< P, Q\>_{\mathrm{tr}} = 0 & \\
\<\psi_k | \phi_\ell \> = 0 & \quad \mbox{for all}\quad k\in[r], \ell\in[s]\,.
\end{align}
\end{lem}

\begin{thm}\label{thm:inertial_bound_projective_packing}
For any graph $G$ with projective packing number $\alpha_p(G)$ and Hermitian weighted adjacency matrix $W$:

\begin{equation}\label{eq:alphaq}
\alpha_p(G) \le n^0(W) + \min\{n^+(W) , n^-(W)\}.
\end{equation}
\end{thm}

\begin{proof}
Let 
\[
P^{(u)} = \sum_{k\in [r^{(u)}]} |\psi^{(u,k)}\>\<\psi^{(u,k)}|
\]
denote the spectral resolution of $P^{(u)}$, where $r^{(u)}$ is its rank. 
Let 
\[
r = \sum_{u\in V} r^{(u)}\,.
\]
Define the composite vectors 
\begin{equation}
|\Psi^{(u,k)}\> = |u\> \otimes |\psi^{(u,k)}\>\,.
\end{equation}
For all $u,v\in V$, $k\in[r^{(u)}]$, and $\ell\in[r^{(v)}]$, we have
\begin{eqnarray}
\<\Psi^{(u,k)}|\Psi^{(v,\ell)}\>                       & = & \delta_{u,v} \cdot \delta_{k,\ell} \\
\<\Psi^{(u,k)}|(W\otimes I_d)|\Psi^{(v,\ell)}\>  & = & 0\,.
\end{eqnarray}
The above equalities hold due to the special tensor product structure of the vectors $|\Psi^{(u)}\>$,  the orthogonality condition in (\ref{eq:projective_orthogonal}), and Lemma~\ref{lem:elementary}.
It follows that these vectors span a $r$-dimensional isotropic subspace with respect to the quadratic form defined by $W\otimes I_d$. 

Using Lemma~\ref{lem:max_iso} and that the inertia of $W\otimes I_d$ is $d$ times the inertia of $W$, we obtain
\[
\frac{r}{d} \le n^0(W) + \min\{ n^+(W) + n^-(W) \}\,,
\]
which completes the proof since this bound holds for all projective packings of $G=(V,E)$ and all Hermitian weighted adjacency matrices $W$ of $G$.
\end{proof}

\section{Eigenvalue upper bound for $\alpha_q(G)$}

Hoffman, in an unpublished paper, proved that for $\Delta$-regular\footnote{We use the unconventional symbol $\Delta$ instead of $d$ for the degree of regular graphs because $d$ is the dimension of the Hilbert space used in the definition of the quantum independence number.} graphs:
\[
\alpha(G) \le \frac {n|\lambda_n|}{\Delta + |\lambda_n|},
\]
where $\lambda_n$ is the smallest eigenvalue of $A$. This result is typically proved using interlacing of the quotient matrix, and is known as the Hoffman bound or ratio bound. 

Lov\'asz (Theorem 9 in  \cite{lovasz79}) proved that for $\Delta$-regular graphs:
\begin{equation}\label{eq:thetaHoffman}
\theta(G) \le \frac {n|\lambda_n|}{\Delta + |\lambda_n|}.
\end{equation}
It is therefore immediate that the Hoffman bound is an upper bound for $\alpha_q(G)$ for regular graphs.

Golubev \cite{golubev17} proved that for \emph{any} graph 
\[
\alpha(G) \le \frac{n(\mu - \delta)}{\mu}\,,
\] 
where $\delta$ is the minimum degree and $\mu$ is the largest eigenvalue of the Laplacian matrix of $G$. This bound equals the Hoffman bound for regular graphs. 

Bachoc \emph{et al} subsequently proved (see section 5 in \cite{bachoc17}), in the context of simplicial complexes, that for any graph:

\[
\theta(G) \le \frac{n(\mu - \delta)}{\mu}.
\] 

It is therefore immediate that the Golubev bound is an upper bound for $\alpha_q(G)$.

\section{Quantum clique number $\omega_q(G)$}

Mancinska and Roberson \cite{mancinska162} also define the quantum clique number where $\omega(G) \le \omega_q(G) = \alpha_q(\overline{G})$. It is straightforward to show that for non-empty graphs $\omega(G) \le \mathrm{rank}(G)$, by interlacing the eigenvalues of $G$ with the eigenvalues of the complete graph on $\omega(G)$ vertices. Alon and Seymour \cite{alon89} used the complement of the folded 7-cube on 64 vertices to demonstrate that $\chi(G) \not \le \mathrm{rank}(G)$, since this graph has $\chi = 32$ and $\mathrm{rank} = 29$. Therefore at some point in the hierarchy of parameters between $\omega$ and $\chi$, the rank of a graph ceases to be an upper bound. We can prove that rank is an upper bound for $\omega_q(G)$ as follows.

\begin{thm}
For any non-empty graph $G$, $\omega_q(G) \le \mathrm{rank}(G)$.
\end{thm}

\begin{proof}
Elphick and Wocjan \cite{elphick17} proved that for any graph $G$, $n - 1 \le n^-(A) + n^-(\overline{A})$. Therefore
\[
\omega_q(G) = \alpha_q(\overline{A}) \le n^0(\overline{A}) + n^+(\overline{A}) = n - n^-(\overline{A}) \le 1 + n^-(A) \le \mathrm{rank}(G).
\]
\end{proof}

The complement of the folded 7-cube has $\alpha = 2$ so $n/\alpha \not \le \mathrm{rank}(G)$, and consequently rank is not an upper bound for the fractional chromatic number. We do not know whether rank is an upper bound for $\theta(\overline{G}).$

\section{Implications for $\alpha_q(G)$  and for $\alpha(G)$}

It follows from Theorem~\ref{thm:inertial_bound} that any graph with $\alpha(G) = n^0 + \min{(n^+ , n^-)}$, has $\alpha_q = \alpha$. This is the case for numerous graphs, including odd cycles, perfect, folded cubes, Kneser, Andrasfai, Petersen, Desargues, Grotzsch, Heawood, Clebsch and Higman-Sims graphs. Furthermore if the inertia bound is tight with an appropriately chosen weight matrix than again $\alpha_q = \alpha$. This is the case for all bipartite graphs. There are also many graphs, including Chvatal, Hoffman-Singleton, Flower Snark, Dodecahedron, Frucht, Octahedron, Thomsen, Pappus, Gray, Coxeter and Folkman   for which $\alpha = \lfloor\theta\rfloor$, so again $\alpha_q = \alpha$. For all such graphs there are no benefits from  quantum entanglement for independence. The Clebsch graph demonstrates that the inertia bound is not an upper bound for $\lfloor \theta'(G) \rfloor$. 

Elzinga and Gregory \cite{elzinga10} asked whether there exists a real symmetric weight matrix $W$ for every graph $G$ such that:

\begin{equation}\label{eq:elzinga}
\alpha(G) = n^0(W) + \min{(n^+(W) , n^-(W))}?
\end{equation}

They demonstrated experimentally that this is true for all graphs with up to 10 vertices, and for vertex transitive graphs with up to 12 vertices. Sinkovic \cite{sinkovic18} subsequently proved that there is no real  symmetric weight matrix for which~(\ref{eq:alpha}) is tight for Paley 17.  This leaves open, however,  whether there is always a Hermitian weight matrix for which~(\ref{eq:alpha}) is exact. 

It follows from Theorem~\ref{thm:inertial_bound}, that every graph with $\alpha < \alpha_q$ is a counter-example to~(\ref{eq:elzinga}) for real symmetric and Hermitian weight matrices. This leads to the question of whether ~(\ref{eq:elzinga}) is true for $\alpha_q$ or $\alpha_p$? It follows from Theorem~\ref{thm:inertial_bound_projective_packing} that the answer is no, because for some graphs, such as the line graph of the cartesian product of $K_3$ with itself, the projective packing number is non-integral.

There are also numerous regular graphs for which the Hoffman bound on $\alpha(G)$ is exact, but the unweighted inertia bound is not. Examples include the Shrikhander, Tesseract, Hoffman and Cuboctahedral graphs. There are also many regular graphs where the floor of the Hoffman bound is exact, but the unweighted inertia bound is not. Examples include some circulant, cubic and quartic graphs. For all of these graphs, $\alpha_q = \alpha$.

It would be interesting to find the graph with the smallest number of vertices that has $\alpha(G) < \alpha_q(G)$. Such a graph must have at least 11 vertices (given the experimental results due to Elzinga and Gregory).  The smallest such graph that we know of is due to Piovesan (see Figure 3.1 in \cite{piovesan16}) which has 24 vertices, with $\chi = \alpha =5, \chi_q = 4$ and $\alpha_q \ge 6$.

\section{An example}
We conclude by using the Clebsch graph to illustrate the parameters and bounds described in this paper. As is well known, the Clebsch graph is strongly regular and vertex transitive with 16 vertices. It has $\omega = 2, \chi = 4, \alpha = 5$ and spectrum $(5^1, 1^{10}, -3^5)$. The inertia bound is therefore exact, so $\alpha_q = \alpha_p = 5$. It also has $\theta' = \theta = \theta^+ = 6$. Not only is the inertia bound exact for $\alpha_q$ but the Hoffman bound in (\ref{eq:thetaHoffman}) is exact for $\theta$. Because the graph is vertex transitive:
\[
\alpha(G) \chi_f(G) = \theta(G) \theta(\overline{G}) = \theta'(G) \theta^+(\overline{G}) =  \theta^+(G) \theta'(\overline{G}) = \alpha_p(G) \xi_f(G) = n,
\]
where $\chi_f$ is the fractional chromatic number and $\xi_f$ is the projective rank. All of these equalities are well known apart from the last one, which is due to Roberson \cite{roberson13}. Therefore $\chi_f = \xi_f = 3.2$ and $\theta'(\overline{G}) = \theta(\overline{G}) = \theta^+(\overline{G}) = 2.67.$ Finally, it is demonstrated in \cite{wocjan18} that $\chi_q = 4$.

\subsection*{Acknowledgements}

We would like to thank David Roberson for helpful comments on an earlier version of this paper, in particular in regard to the projective packing number. We would also like to thank David Anekstein for testing various ideas for this paper experimentally.

This research has been supported in part by National Science Foundation Award 1525943.

\end{document}